\documentclass[a4paper]{amsart}

\usepackage{subcaption}
\usepackage{graphicx}
\usepackage{xcolor}
\usepackage{colortbl}    
\def\pdfnero{0} 

\if\pdfnero1
    \pagecolor[rgb]{0,0,0} 
    \color[rgb]{1,1,1}
    \definecolor{citeColor}{rgb}{0.9,0.9,0.9}
    \definecolor{linkColor}{rgb}{0.9,0.9,0.9}

    \definecolor{mioara}{rgb}{0,0.5,1}
\else
    \definecolor{citeColor}{rgb}{0.1,0.1,0.1}	
    \definecolor{linkColor}{rgb}{0.1,0.1,0.1}

    \definecolor{mioara}{rgb}{1,0.5,0}
\fi

\definecolor{urlColor}{rgb}{0,0,0.8}				

\definecolor{azzurro}{rgb}{0.9,1,1}    
\definecolor{grigiomio}{rgb}{0.9,0.9,0.9}    
\definecolor{bluu}{rgb}{0.7,1,1}    

\usepackage[utf8]{inputenc}
\usepackage[T1]{fontenc}
\usepackage[english]{babel}

\usepackage{amsfonts}
\usepackage{amsmath}
\usepackage[alphabetic]{amsrefs}
\usepackage{amssymb}
\usepackage{amstext}
\usepackage{amsthm}
\usepackage[all, knot]{xy}

\usepackage{algorithm}
\usepackage{algpseudocode}

\usepackage{verbatim}

\algrenewcommand\algorithmicrequire{\textbf{Input:}}
\algrenewcommand\algorithmicensure{\textbf{Output:}}

\usepackage{geometry}
\usepackage{fullpage}

\usepackage[shortlabels]{enumitem}
\usepackage{graphicx}
\usepackage{hyperref}
\usepackage{listings}
\usepackage{mathrsfs}
\usepackage{mathtools}
\usepackage{multicol}
\usepackage{stmaryrd}

\usepackage{tikz}
\usepackage{tikz-cd}
\usetikzlibrary{arrows,backgrounds,chains,decorations.pathreplacing,patterns,shapes, calc, positioning}

\newcommand{\inv}{\mathsf{inv}}
\newcommand{\Inv}{\mathsf{Inv}}
\newcommand{\Des}{\mathsf{Des}}
\newcommand{\des}{\mathsf{des}}
\newcommand{\lec}{\mathsf{lec}}
\newcommand{\single}{\mathsf{single}}

\newcommand{\sw}{\mathsf{sw}}
\newcommand{\srw}{\mathsf{srw}}
\newcommand{\ins}{\mathsf{ins}}
\newcommand{\rins}{\mathsf{rins}}
\newcommand{\h}{\mathsf{h}}
\newcommand{\wv}{\mathsf{wv}}

\newcommand{\rwave}{\mathsf{rwave}}
\newcommand{\conc}{\mathsf{conc}}
\newcommand{\srdes}{\mathsf{srdes}}

\newcommand{\Hess}{\mathcal Hess}

\newcommand{\cX}{\mathcal X}

\newcommand{\bP}{\mathbb P}
\newcommand{\bS}{\mathbf {S}}

\newtheorem{theorem}{Theorem}[section]

\newtheorem{proposition}[theorem]{Proposition}
\newtheorem{corollary}[theorem]{Corollary}

\theoremstyle{definition}
\newtheorem{definition}[theorem]{Definition}

\newtheorem{example}[theorem]{Example}

\theoremstyle{remark}
\newtheorem{remark}[theorem]{Remark}

\numberwithin{equation}{section}

\title{Generalizing Eulerian numbers via semipermutations:  topological  and combinatorial aspects}

\author{Giovanni Gaiffi}
\address{Università di Pisa\\Dipartimento di Matematica\\ Largo Bruno Pontecorvo 5, 56127 Pisa\\ Italy}\email{giovanni.gaiffi@unipi.it}

\author{Giovanni Interdonato}
\address{Scuola Normale Superiore\\
	Piazza dei Cavalieri 7,
	56126 Pisa\\ Italy}\email{giovanni.interdonato@sns.it }

\begin{document}
	
\begin{abstract}
In a paper by Lin an interesting family of semipermutations comes out to index the elements of a cohomology basis of a Hessenberg type variety. The corresponding Betti numbers are a generalization of Eulerian numbers. We show three different subsets of the symmetric group that are in bijection with the set of these semipermutations. These bijections  preserve the statistics \(\lec\) and \(\des\): one of these is first obtained by  an algebraic-topological argument, and all of them are explicitly described in combinatorial terms.
\end{abstract}

\maketitle
\tableofcontents

\section*{Introduction}
In \cite{Lin24} the author points out the following family of semipermutations, more precisely the set $\mathcal{S}_{n,k+1}$ of permutations of $(k+1)$ different numbers chosen from $[n]:=\{1,\cdots,n\}$, equipped with a descent statistic. 

 \begin{definition}	 

The set $\mathcal{S}_{n,k+1}$ is made of all pairs $(\alpha_0,\rho)$, where $\rho=\rho_1\cdots\rho_{k+1}$ is a word made by $k+1$ different numbers of $[n]$, and $\alpha_0=[n]\setminus\{\rho_1,\dots,\rho_{k+1}\}$.

Given a pair $(\alpha_0,\rho)$ as above, we define its  {\em descent set} as $$\Des((\alpha_0,\rho))=\{(\rho_i>\rho_{i+1})\mid i\in[k]\}\sqcup\{(a>\rho_1)\mid a\in\alpha_0 \}$$ and we will denote by $\des((\alpha_0,\rho))=|\Des((\alpha_0,\rho))|$ the cardinality of this set.
\end{definition}  

\begin{example}
Let $(\alpha_0,\tau)=(\{4,6,7\},52318)\in \mathcal{S}_{8,5}$.
We have that $$\Des((\{4,6,7\},52318))=\{(5>2),(3>1)\}\sqcup\{(6>5),(7>5)\}.$$
Therefore $\des((\{4,6,7\},52318))=2+2=4$.
\end{example}

The number of    permutations of $k+1$ different numbers in $[n]$ with $i$ descents appears naturally in \cite{Lin24} as the \(2i\)-th Betti number of a Hessenberg type variety.
The author notices  that these Betti numbers are a  natural generalization of the Eulerian numbers (when \(k=n-2\) we are dealing with  the permutations of \(n-1\) different numbers in $[n]$ with \(i\) descents, which is equivalent to dealing with the permutations of \(n\) numbers with \(i\) descents, whose cardinality is  the Eulerian number $A(n,i+1)$) and observes  that it would   be interesting to study further properties of these numbers (see also this related question on MathOverflow \cite{MOF25} and the OEIS sequence \cite{OEIS_A381706}).

In the present paper we explore this topic, providing a further topological  interpretation of these numbers.  This gives the inspiration to then find   three explicit bijections  of the set of permutations of $k+1$ different numbers in $[n]$ with $i$ descents with some sets of permutations considered with the statistics \(\lec\) or \(\des\). One of the involved bijections turns out to be  an explicit presentation of a bijection introduced by Foata and Han in \cite{FH08}.

The structure of this paper is the following. After recalling  in Section \ref{secpreliminary} the statistics \( \des\) and \(\lec\),  in Section \ref{secalgebraictopology} we consider the Hessenberg type variety whose cohomology in degree \(2i\), according to Lin's construction,  admits a basis indexed by the    permutations of $k+1$ different numbers in $[n]$ with $i$ descents.
We notice that this variety is isomorphic to a De Concini--Procesi wonderful model (these models were introduced in the seminal paper \cite{wonderful1}): a different construction of a cohomology basis coming from the theory of wonderful models then provides us with an algebraic-topological proof of a bijection between the set of the permutations of $k+1$ different numbers in $[n]$ with $i$ descents and the set of the elements of \(\mathcal{S}_n\) whose first hook from the right in the hook factorization has cardinality at least \(n-k\) and whose \(\lec\) is equal to \(i\).

Motivated by this,  we show an explicit combinatorial presentation of the above mentioned bijection. We do this in two steps. First in   Section \ref{sec:reversewaves} we deal with the case \(k=n-2\), i.e. we provide a bijection of \(S_n\) which preserves the statistics \(\lec\) and \(\des\). This is based on the idea of {\em reverse wave}:  let $w=w_1\cdots w_{\ell}$ be a non-empty word made of distinct numbers. Let $c$ be the position in $w$ of the maximum element of $w$, then  $w$ is a \textit{reverse wave} if $c\neq\ell$ and 
    \begin{equation}
    w_1<w_2<\dots<w_{c-1}<w_{\ell}<w_{\ell-1}<\dots<w_c.
    \end{equation}
    or if $c=\ell$ and the elements are in increasing order. 
    
    Then in Section \ref{sec:generalexplicit} we extend this bijection to the general case (i.e. for  any  value of \(k\)).

Finally, in Section \ref{sec:foata} we present a different  bijection $\psi$ of the symmetric group $\mathcal{S}_n$ that sends the $\des$ statistic to the $\lec$ statistic.  This bijection $\psi$ is explicitly presented and appears to be equal, unless we compose with easy bijections, to the map $F$ defined by Foata and Han in \cite{FH08}*{Section~6}). It is based on the idea of {\em wave}:  let $w=w_1\cdots w_{\ell}$ be a non-empty word made of distinct numbers. Let $c$ be the position in $w$ of the maximum element of $w$, we say that $w$ is a \textit{wave} if $c\neq\ell$ and 
    \begin{equation}
        w_{\ell}<w_{\ell-1}<\dots<w_{c+1}<w_1<w_2<\dots<w_c.
    \end{equation}
    or if $c=\ell$ and the elements are in increasing order. 
    
  Also this bijection can be extended to the general case so in the end of the section  we obtain  Theorem \ref{thm:bigezionek} which shows three different sets of \(S_n\) whose cardinality is equal to the cardinality of the set of the permutations of $k+1$ different numbers in $[n]$ with $i$ descents. This extends our first algebraic-topological result, and therefore gives an explicit insight on the semipermutations pointed out by Lin.
  
  {\em Acknowledgments} The authors would like to thank Michele D’Adderio for helpful discussions on this issue and for his valuable advice.

\section{Some preliminary notation}
\label{secpreliminary}
If $w=w_1\dots w_{\ell}$ is a word, we define its descent set as $\Des(w)=\{(w_i>w_{i+1})\mid i\in[\ell-1]\}$ and we will denote by $\des(w)=|\Des(w)|$ the cardinality of this set. We define its inversion set as $\Inv(w)=\{(i,j)\mid1\leq i<j\leq\ell,\, w_i>w_j\}$ and we will denote by $\inv(w)=|\Inv(w)|$ the cardinality of this set. These definitions naturally extend to permutations, considering the word obtained by writing them in one-line notation.

\begin{example}\label{ex_sigma}
Let $\sigma=45762318\in \mathcal{S}_8$. We will have that $\Des(\sigma)=\{(7>6),(6>2),(3>1)\}$ and hence that $\des(\sigma)=3$.
\end{example}

Given a word $w=w_1\cdots w_{\ell}$ we recall the definition of Hook factorization, first introduced by Gessel in \cite{G91}. There exists and is unique a decomposition of $w=\gamma*\alpha_1*\cdots*\alpha_k$, where the operator $*$ is the concatenation of words, such that the elements in $\gamma$ are in increasing order,  and each $\alpha_i$ is a \textit{hook}, that is $\alpha_i=a^{(i)}_1\cdots a^{(i)}_{s_i}$, with $s_i\geq2$, $a^{(i)}_1>a^{(i)}_2$ and $a^{(i)}_2<\dots<a^{(i)}_{s_i}$, for each $i=1,\dots,k$. We then define the statistic $\displaystyle\lec(w)=\sum_{i=1}^k\inv(\alpha_i)$ as the sum of the number of inversions of each hook.
We will also denote a word with elements in increasing order, such as $\gamma$, as a \textit{trivial hook}. 

\begin{remark}\label{ind_lec}
We can obtain this factorization recursively, taking as $\alpha_k$ the suffix of $w$ that starts at the position of the rightmost descent of $w$, and then iterating this reasoning on the remaining prefix obtained by removing $\alpha_k$. The process ends when no further descents are present, i.e. when we have a word whose elements are in increasing order: this will be $\gamma$.
\end{remark}

\begin{example}\label{ex_tau}
Let $\tau=18326457\in \mathcal{S}_8$. Its Hook factorization is:
\begin{equation}
\tau=\underline{18}*32*6457,
\end{equation}
where $\gamma=18$ is the underlined factor. We therefore have that 
$\lec(\tau)=\inv(32)+\inv(6457)=1+2=3$.
\end{example}

\section{A bijection pointed out by  algebraic topology}
\label{secalgebraictopology}
Let us denote by $\cX$ the permutohedral variety  of dimension $n-1$, i.e. the toric variety associated to the root system of type \(A_{n-1}\). 
As it is well known (see for instance \cite{Pro90},\cite{Lin24}*{Section~2}), it can also be obtained as an iterated blowup of $\bP^{n-1}$:
\[
\xymatrix{
	\cX=\cX_{n-2}\ar[r]^{\pi_{n-2}} &\cX_{n-3} \ar[r]^{\pi_{n-3}}& \cdots \ar[r]^{\pi_{2}}&
	\cX_{1} \ar[r]^{\pi_{1}\ \ \ \ \ } & \cX_0 = \bP^{n-1}
}
\]
where  each $\cX_{k+1}$ is the blowup of $\cX_k$ along the strict transform of the $k$-dimensional coordinate spaces ($0\le k\le n-3$).

In \cite{Lin24}*{Section~3}  the author describes an explicit isomorphism between  the permutohedral variety  of dimension $n-1$, and the regular semisimple Hessenberg variety $\Hess(\bS,h_+)$ associated to the Hessenberg function $h_+$ defined by $h_+(i)=i+1$ for $1\le i\le n-1$ and $h_+(n)=n$.
In this process, he further obtains isomorphisms from $\cX_k$,the $n-1$ dimensional variety that is the $k$-th step in the iterated blowups, to a ``Hessenberg type'' subvariety of a partial flag variety which is denoted by $\Hess^{(k+1)}(\bS,h_+)$.
Then he exhibits a basis for the integer cohomology $H^*(\cX_k)$ whose elements are in one-to-one correspondence with the permutations of $(k+1)$ different numbers chosen from $[n]:=\{1,\cdots,n\}$. This allows for a combinatorial computation of the Betti numbers of $\cX_k$, more precisely  the even Betti numbers of $\cX_k$ are obtained as follows (the odd Betti numbers of $\cX_{k}$ are all equal to $0$).

\begin{theorem}\label{thm:Lin} [see \cite{Lin24}]
The $2i$-th Betti number of $\cX_{k}$ is given by
	\begin{align*}
		\beta_{2i}(\cX_{k})&= \# (\text{permutations of $k+1$ different numbers in $[n]$ with $i$ descents})
	\end{align*}
\end{theorem}
These Betti numbers are a quite natural generalization of the Eulerian numbers. In fact  we notice that when \(k=n-2\) the formula  above states that $\beta_{2i}(\cX_{n-2})$ is the number of permutations of \(n-1\) different numbers in $[n]$ with \(i\) descents, which is equivalent to saying  that it is the number of permutations of \(n\) numbers with \(i\) descents, i.e. the Eulerian number $A(n,i+1)$. On the geometrical side, we observe that since $\cX_{n-2}$ is the permutohedral variety, it is well known that its Betti numbers $\beta_{2i}(\cX_{n-2})$ coincide with  the Eulerian numbers. 

Let now \(k\) be an integer such that \(0\leq k\leq n-2\) and let us now consider in $\mathbb C^n$ the  arrangement \(\mathcal G_k\) given by all the coordinate subspaces of dimension \(\leq k\). This arrangement is built according to the definition given  by De Concini and Procesi in \cite{wonderful1};  the space that we obtain by blowing up all the subspaces in \(\mathcal G_k\) is a wonderful model $Y_{k}$ whose cohomology ring is isomorphic to  the cohomology of the corresponding projective model   i.e. $\cX_{k}$.  In  ~\cite{gaiffiselecta} a description of  a monomial basis  for the integer cohomology of De Concini--Procesi wonderful models is provided, and this basis for $Y_k$ can be described in the following way.

A monomial in this basis is a product of Chern classes of  (transforms of)  subspaces in \(\mathcal G_k\). Let \(\zeta_{G_1}^{\alpha_1}\zeta_{G_2}^{\alpha_2} \cdots \zeta_{G_s}^{\alpha_s} \) be such a monomial. Then
\begin{enumerate}
\item \(G_1,G_2,\ldots, G_s\) is a chain in \(\mathcal G_k\) with respect to inclusion (say \(G_1\subset \cdots \subset G_s\)) where in each inclusion the codimension is greater than 1 and the codimension of \(G_s\) in $\mathbb C^n$ is greater than 1;
\item for every \(i\), \(\alpha_i\) is a positive integer less than the codimension of \(G_i\) in \(G_{i+1}\) (for \(i=s\) we consider the codimension of \(G_s\) in $\mathbb C^n$). 
\end{enumerate} 

We now use  for brevity a notation as in the following example: if \(G_1\) is the subspace given by the equations \(x_1=x_2=x_4=x_6=0\) then we associate to \(G_1\) the subset \({1,2,4,6}\) of \([n]\), we denote  \(G_1=H_{1,2,4,6}\)  and we call \(\zeta_{1,2,4,6}\) its associated Chern class. According to this notation,  there is a bijection between the elements in \(\mathcal G_k\) and the subsets of \([n]\) whose  cardinality is greater than or equal to  \(n-k\).
\begin{remark}
As an example let \(k=7\) $n=11$ and consider the monomial
\begin{equation}
\label{esempiobase}
\zeta_{\{1,2,4,5\}}^2\zeta_{\{1,2,4,5,6,7,8\}}^2\zeta_{\{1,2,4,5,6,7,8,9,11\}}
\end{equation}  
which is an element of the basis of the cohomology of \(\cX_{7} \).

For instance, notice that here  the exponent of $\zeta_{\{1,2,4,5\}}$ is strictly less than $4$, which is  the codimension of $H_{1,2,4,5}$ in $\mathbb C^{11}$. 
\end{remark}
We now show an algorithm that produces a bijection between this monomial  basis of $\cX_{k}$ and the elements of \(\mathcal{S}_n\) whose first hook from the right in the hook factorization has cardinality at least \(n-k\) (see also \cite{GPS24}*{Section~5.2} for the case $k=n-2$). This bijection is grade-preserving provided that we consider in \(\mathcal{S}_n\) the grade induced by the statistic $\lec$.
We illustrate the algorithm producing a permutation \(\sigma\) from the above mentioned monomial 
\[
\zeta_{\{1,2,4,5\}}^2\zeta_{\{1,2,4,5,6,7,8\}}^2\zeta_{\{1,2,4,5,6,7,8,9,11\}}
\] 
\begin{itemize}
\item We first look at the elements in $\{1,\ldots,11\}$ that do not appear in the description of the monomial. We write them in increasing order obtaining the non-hook part $\gamma$ of $\sigma$. In our example we have $\{1,\dotsc,11\}\setminus \{1,2,4,5,6,7,8,9,11\}=\{3,10\}$ so $\gamma=3\  10$.
\item The first (from the right) hook  of $\sigma$ is obtained by taking  the numbers 1,2,4,5, that belong to the smallest set of the chain, and by arranging them in such a way that the number of inversions is equal to the exponent 2 of \(\zeta_{\{1,2,4,5\}}^2\). So we get 4125.
\item The second (from the right) hook  of $\sigma$ is formed using the numbers of the second set of the chain, i.e. \(\{1,2,4,5,6,7,8\}\), that do not appear in the smallest one. In the example those numbers are $\{6,7,8\}$ which is $\{1,2,4,5,6,7,8\}\setminus \{1,2,4,5\}$, arranged in such a way as to obtain 2 inversions since $2$ is the exponent of $\zeta_{\{1,2,4,5,6,7,8\}}$ in the monomial: 867.
\item Analogously, the last (from the right) hook  is 11 \( \) 9
\end{itemize}
In the end we obtain $\sigma = 3 \  10 \  11 \  9 8674125$, which satisfies  $\lec(\sigma)=5$.

From this bijection we deduce another way to obtain the Betti numbers of $\cX_{k}$:
\begin{theorem}
Let $0\leq k\leq n-1$ and $i>0$. The $2i$-th Betti number of $\cX_{k}$ is given by
\begin{align*}
\beta_{2i}(\cX_{k})= \# (&\text{elements \(\sigma\in \mathcal{S}_n\), such that  \(\lec(\sigma)=i\), whose  first}\\
&\text{hook from the right has cardinality at least \(n-k\)})
\end{align*}
\end{theorem}

Therefore considering two different cohomology bases of  $\cX_{k}$ allows us to deduce, by  double counting, the following result:
\begin{corollary}
\label{corollario:bigezionek}
Let $0\leq k\leq n-1$ and $i\geq0$. The cardinality of the set of the permutations of $k+1$ different numbers in $[n]$ with $i$ descents is equal to the cardinality of  the set of the  elements \(\sigma\in \mathcal{S}_n\), such that  \(\lec(\sigma)=i\),  whose  first hook from the right 
 has cardinality at least \(n-k\).
\end{corollary}

\begin{remark}
    The statement of the Corollary \ref{corollario:bigezionek} includes also the cases $i=0$ and $k=n-1$. Namely if $i=0$ we have only one permutation of $k+1$ different numbers in $[n]$ with no descents, and only one permutation $\sigma\in\mathcal{S}_n$ such that $\lec(\sigma)=0$, that is the identity. 
    
    If $k=n-1$ the statement immediately follows from the case $k=n-2$. Indeed, since any hook has cardinality at least $2$ by definition, the requests that the first hook from the right has cardinality at least $1$ or $2$ are equivalent, and as we have already observed, the number of permutations of \(n-1\) different numbers in $[n]$ with \(i\) descents is equal to the number of permutations of \(n\) numbers with \(i\) descents.
\end{remark}

\section{An explicit  bijection of \texorpdfstring{$\mathcal{S}_n$}{Sn} preserving des-lec statistic: reverse waves}
\label{sec:reversewaves}
Motivated by Corollary~\ref{corollario:bigezionek} we now search for a combinatorial bijection between the set of the permutations of $k+1$ different numbers in $[n]$ with $i$ descents and the set of the  elements \(\sigma\in \mathcal{S}_n\), such that  \(\lec(\sigma)=i\),  whose  first hook from the right has cardinality at least \(n-k\).
In this section we  focus on the case \(k=n-2\), i.e. we search for a combinatorial bijection of the symmetric group $\mathcal{S}_n$ that sends the $\des$ statistic to the $\lec$ statistic. 

To do so, we will present an explicit bijection of the set $$W=\{a_1\cdots a_{\ell}\mid \ell\in\mathbb{N},\,a_i\in\mathbb{N}^*\,\forall i\text{ and }a_i\neq a_j\,\forall\, i\neq j\}$$ of all finite words with distinct entries, such that if we restrict to the symmetric group $\mathcal{S}_n$, it becomes a bijection of $\mathcal{S}_n$. 

\begin{remark} \label{lec_inverse}
We can evaluate the $\lec$ of a word $w$ in a recursive way, as suggested by the Remark~\ref{ind_lec}. If $w=\gamma*\alpha_1*\cdots*\alpha_k$ is the Hook factorization, and $k=0$, then $\lec(w)=\lec(\gamma)=0$, while if $k>0$, then we have by definition that
\begin{equation}
    \lec(w)=\lec(\gamma*\alpha_1*\cdots*\alpha_{k-1})+\lec(\alpha_k).
\end{equation}

If a hook $\alpha$ has length $\ell$, then we obtain that $1\leq\lec(\alpha)\leq\ell-1$. Moreover, if we have $\ell$ distinct numbers $a_1<a_2<\cdots<a_{\ell}$, for any $0\leq d\leq\ell-1$, we have exactly one way to rearrange them to create a hook, possibly trivial if $d=0$,  $\alpha_d$ with $\lec(\alpha_d)$ equal to $d$, that is, $\alpha_d=a_{d+1}a_1a_2\cdots a_da_{d+2}\cdots a_{\ell}$.
\end{remark}

We will define a similar way of evaluating descents by decomposition in smaller words.

\begin{definition}
    Let $w=w_1\cdots w_{\ell}\in W$ be a non-empty word of length $\ell$. Let $c$ be the position in $w$ of the maximum element of $w$, we say that $w$ is a \textit{reverse wave} if $c\neq\ell$ and 
    \begin{equation}
    w_1<w_2<\dots<w_{c-1}<w_{\ell}<w_{\ell-1}<\dots<w_c.
    \end{equation}
    We will also denote a word with elements in increasing order, that is the case with $c=\ell$, as a \textit{trivial reverse wave}.
\end{definition}

\begin{remark}\label{rem:pref_rw}
    If $w$ is a reverse wave, then by definition, any prefix of $w$ is also a reverse wave, possibly trivial.
\end{remark}

\begin{remark}\label{rem:unique_rwave}
Similarly to what was observed in Remark~\ref{lec_inverse}, if a reverse wave $rw$ has length $\ell $, then we have that $1\leq\des(rw)\leq\ell-1$. Moreover if we have $\ell$ distinct numbers $b_1<b_2<\cdots<b_{\ell}$, for any $0\leq d\leq\ell-1$, we have exactly one way to rearrange them to create a reverse wave, possibly trivial if $d=0$, $r_d$ with $\des(rw_d)$ equal to $d$, that is, ${rw_d=b_1b_2\cdots b_{\ell-d-1}b_{\ell}b_{\ell-1}\cdots b_{\ell-d}}$.
\end{remark}

\begin{definition}
    \label{def:special_reverse_wave}
    Let $w=w_1\cdots w_{\ell}\in W$ be a non-empty word of length $\ell $. The \textit{special reverse wave} of $w$ is the prefix $\srw(w)=w_1w_{2}\cdots w_t\cdots w_{s-1}w_s$ of $w$ where:
    \begin{itemize}
        \item $1\leq t<\ell$ is the index of the first descent of $w$, so it is the minimum such that $w_t>w_{t+1}$. If $w$ has no descents, if and only if elements of $w$ are in increasing order, then $\srw(w)=w$.
        \item $t\leq s\leq\ell$ is the maximum such that $w_{t}>w_{t+1}>\dots>w_{s-1}>w_s>w_{t-1}$, where, if $t=1$, we consider $w_0=0$. It exists since $w_{t}>w_{t-1}$, by minimality of $t$, and if $t=1$, then is it true that $w_{t}>0=w_{t-1}$ by definition of $W$.
    \end{itemize}  
\end{definition}

\begin{remark}\label{rem:special_reverse_wave}
    If $\srw(w)=w_1\cdots w_t\cdots w_{s-1}w_s$ is the special reverse wave of $w$, then we have 
    \begin{equation}\label{eq:srw_order}
    w_1<w_2<\dots<w_{t-1}<w_s<w_{s-1}<\dots<w_{t}
    \end{equation}
    and in particular that $\srw(w)$ is a reverse wave, possibly trivial. This is the maximal prefix of $\sigma$ that is also a reverse wave.  If $w$ has no descents, then $\srw(w)=w$, and in this case $\srw(w)$ is a trivial reverse wave.
    
    We also have, if $s\neq\ell$, that $w_{s+1}>w_s$ or that $w_{s+1}<w_{t-1}<w_{s}$.
\end{remark}

We will denote by $w\setminus\srw(w)=w_{s+1}\cdots w_{\ell}$ the suffix obtained by removing from $w$ its special reverse wave. If $\srw(w)=w$ then $w\setminus\srw(w)$ is the empty word.

\begin{example}\label{ex_srw}
Let $\sigma=45762318\in \mathcal{S}_8\subset W$, then its special reverse wave is $\srw(\sigma)=4576$ and $w\setminus\srw(w)=2318$.
\end{example}

We introduce a new statistic that will allow us to correlate the number of the descents of a word with those of its special reverse wave.

\begin{definition}
    \label{def:special_reverse_des}
    Let $w=w_1\cdots w_{\ell}\in W$ be a non-empty word and $\srw(w)=w_1w_{2}\cdots w_t\cdots w_{s-1}w_s$ be its special reverse wave. We denote \textit{the number of special reverse descents} of $w$ by
    \begin{equation}
        \srdes(w)=\des(\srw(w))+\chi(w_s>w_{s+1})
    \end{equation}
    where $\chi(P)$ equals $1$ if the property $P$ is true, and $0$ otherwise, and if $s=\ell$ then $\chi(w_s>w_{s+1})=0$.
\end{definition}

\begin{remark}\label{rem:des_rspezza}
We have that $\des(w)=\srdes(w)+\des(w\setminus\srw(w))$. Indeed if $s=\ell$  we obtain that $w=\srw(w)$ and $\srdes(w)=\des(\srw(w))=\des(w)$ as wanted, while if $s<\ell$, we get that
\begin{align}
    \des(w)=\des(w_1\cdots w_{\ell})&= \des(w_1\cdots w_{s})+\chi(w_{s}>w_{s+1})+\des(w_{s+1}\cdots w_l)\\
    &=\srdes(w)+\des(w\setminus\srw(w))\notag.
\end{align}
\end{remark}

\begin{example}\label{ex:srw_spezza}
Let $\sigma=45762318\in \mathcal{S}_8\subset W$, as in the previous example. We have that $\srdes(w)=\des(4576)+\chi(6>2)=1+1=2$ and $\des(w)=3=2+1=\srdes(w)+\des(2318)$.
\end{example}

\begin{proposition}\label{prop:rdes_not_0}

Let $w=w_1\cdots w_{\ell}\in W$ be a word that is not increasing and $\srw(w)=w_1\cdots w_s$ its special reverse wave. Then we have that
\begin{equation}\label{eq:des_spezza2}
1\leq\srdes(w)\leq s-1.
\end{equation}   
Moreover if $w$ is an increasing word then $\srdes(w)=0$.
\end{proposition}

\begin{proof}
First of all let us prove that $\srdes(w)\geq1$ if $w$ is not increasing. If $\srdes(w)=0$ then $\des(\srw(w))=0$ and so the first descent of $w$ is exactly the one between the last element of $\srw(w)$ and the first of $w\setminus\srw(w)$. This implies that $\srdes(w)=1$, that is an absurd.   

Let us now demonstrate the other inequality.
We have that $\des(\srw(w))\leq s-1$ and so $\srdes(w)=\des(\srw(w))+\chi(w_s>w_{s+1})\leq s$ and the equality holds if and only if  $\srw(w)$ is a decreasing word and $w_s>w_{s+1}$. This is an absurd, since if $\srw(w)$ is a decreasing word and $w_{s+1}$ is not in the special reverse wave of $w$, then $w_s<w_{s+1}$.

If $w$ is an increasing word then $\srw(w)=w$ and $\srdes(w)=\des(w)=0$ as wanted.
\end{proof}

\begin{proposition}\label{prop:rinv}
Let $a=a_1\cdots a_p$ and $b=b_1\cdots b_q$ two words in $W$ with distinct entries and $1\leq d\leq q-1$. Then there exists a unique word $\rins(a,b,d)$ such that $\srw(\rins(a,b,d))$ is a rearrangement of the elements of $b$ and
\begin{align}\label{eq:rins1}
    \rins(a,b,d)\setminus\srw(\rins(a,b,d))&=a,\\
\srdes(\rins(a,b,d))&=d.\label{eq:rins2}
\end{align}
Moreover for all $w\in W$ that are not increasing,
\begin{equation}\label{eq:rins3}
    \rins(w\setminus\srw(w),\srw(w),\srdes(w))=w.
\end{equation}
\end{proposition}

\begin{proof}
    Let $B=\{\beta_1<\dots<\beta_q\}$ the set of the elements of $b$. 
We will prove that
\begin{equation}
    \mathsf{r}(a,b,d)=\begin{cases}
    \rwave(b,d)*a=\beta_1\cdots\beta_{q-d-1}\beta_q\cdots\beta_{q-d}a_1\cdots a_p & \text{if } \beta_{q-d} <a_1 \text{ or } p=0 \\
    \rwave(b,d-1)*a=\beta_1\cdots\beta_{q-d}\beta_q\cdots\beta_{q-d+1}a_1\cdots a_p & \text{else},
\end{cases}
\end{equation}
is the only word that satisfies the required properties, where $\rwave(b,x)$ is the only reverse wave with the same elements of $b$ and $0\leq x\leq q-1$ descents (see Remark~\ref{rem:unique_rwave}).

We have that both $\rwave(b,d)$ and $\rwave(b,d-1)$ are reverse waves of length $q$, so the special reverse wave of $\mathsf{r}(a,b,d)$ has length at least $q$. Moreover in $\rwave(b,d)$ there is at least one descent, since $d\geq1$, and if $\beta_{q-d}<a_1$ or $p=0$, then $\mathsf{r}(a,b,d)=\rwave(b,d)$. In these cases it is clear that $\mathsf{r}(a,b,d)\setminus\srw(\mathsf{r}(a,b,d))=a$ and that
\begin{equation*}
    \srdes(\mathsf{r}(a,b,d))=\des(\rwave(b,d))+\chi(\beta_{q-d}>a_1)=d+0=d.
\end{equation*}
If $p\geq1$ and $\beta_{q-d}>a_1$ then the position of the first descent of $\mathsf{r}(a,b,d)$ is $q-d+1$ and since $\beta_{q-d}>a_1$, then $\mathsf{r}(a,b,d)=\rwave(b,d-1)$.
In this case it is clear that $\mathsf{r}(a,b,d)\setminus\srw(\mathsf{r}(a,b,d))=a$ and that
\begin{equation*}
    \srdes(\mathsf{r}(a,b,d))=\des(\rwave(b,d-1))+\chi(\beta_{q-d}>a_1)=d-1+1=d.
\end{equation*}
So in every case $\mathsf{r}(a,b,d)$ satisfies the properties of the proposition. Let us now demonstrate that it is the only possible word, and so that $\rins(a,b,d)=\mathsf{r}(a,b,d)$.

Since $\srw(\rins(a,b,d))$ is a rearrangement of the elements of $b$, we obtain that 
\begin{equation*}
\srw(\rins(a,b,d))=\rwave(b,y)=\beta_1\cdots\beta_{q-y-1}\beta_q\cdots\beta_{q-y}    
\end{equation*} with a proper $0\leq y\leq q-1$, and thanks to equation~\eqref{eq:rins1} that $\rins(a,b,d)=\rwave(b,y)*a$. 
Moreover, if $p=0$ equation~\eqref{eq:rins2} implies that
\begin{align*}
    d=\srdes(\rins(a,b,d))=\des(\srw(\rins(a,b,d)))=\des(\rwave(b,y))=y
\end{align*}
and so that $\rins(a,b,d)=\mathsf{r}(a,b,d)$.

If $p>0$, equation~\eqref{eq:rins2} implies that
\begin{align*}
    d&=\srdes(\rins(a,b,d))=\des(\srw(\rins(a,b,d)))\\
    &=\des(\rwave(b,y))+\chi(\beta_{q-y}>a_1)=y+\chi(\beta_{q-y}>a_1),
\end{align*}
and so that $y=d-1$ or $y=d$. 

If $y=d-1$, then the position of the first descent in $\rins(a,b,d)$ is $q-y=q-d+1$, and since $a_1$ is not in the special reverse wave of $\rins(a,b,d)$, we have that $a_1<\beta_{q-d}$, and so that $\rins(a,b,d)=\mathsf{r}(a,b,d)$.

If $y=d$, then we have that $\chi(\beta_{q-y}>a_1)=\chi(\beta_{q-d}>a_1)=0$ and so that $\beta_{q-d}<a_1$. Then, also in this case, $\rins(a,b,d)=\mathsf{r}(a,b,d)$ as wanted.

To prove the identity of the equation~\eqref{eq:rins3} we observe that by the Proposition~\ref{prop:rdes_not_0} $1\leq\srdes(w)\leq \ell(\srw(w))-1$ and so that $\rins(w\setminus\srw(w),\srw(w),\srdes(w))$ is well-defined. Moreover $w$ is a word that satisfies the requested properties, and thanks to the just proved uniqueness of $\rins$, we have that \begin{equation*}
    \rins(w\setminus\srw(w),\srw(w),\srdes(w))=w.
\end{equation*} 
\end{proof}

Now we can define the desired bijection.

Let $\theta:W\to W$ be the map such that:
\begin{itemize}
    \item If $w=\varepsilon$ is the empty word then $\theta(\varepsilon)=\varepsilon$.
    \item Otherwise
    \begin{equation}
        \theta(w)=\theta(w\setminus\srw(w))*\delta_d(\srw(w))
    \end{equation}
    where the $*$ operator is the concatenation of words and $\delta_d(\srw(w))=c_{d+1}c_1c_2\cdots c_dc_{d+2}\cdots c_m$ where $c_1<c_2<\cdots<c_m$ are the elements of $\srw(w)$ and $d=\srdes(w)$, is the unique rearrangement of $\srw(w)$ that is a hook with $\lec$ equal to $d$.
\end{itemize}

\begin{theorem}\label{thm:bijection_2}
    The map $\theta:W\to W$ just defined is a well defined bijection of $W$, such that \begin{equation}
        \des(w)=\lec(\theta(w)).\label{eq:des_lec_2}
        \end{equation}

    Moreover the restriction to the symmetric group $\theta_{\mathcal{S}_n}:\mathcal{S}_n\to \mathcal{S}_n$ is a bijection.
\end{theorem}

\begin{proof}
We will prove by induction on the length $\ell$ of $w\in W$ that $\theta(w)$ is a rearrangement of $w$, and that it is a well defined map. If $\ell=0$ then $w=\varepsilon$ is the empty word and by definition $\theta(\varepsilon)=\varepsilon$, that is fine. If $w$ has length $\ell>0$, then by definition $w\setminus\srw(w)$ is a word of length strictly less than $\ell$, and so by the inductive hypothesis $\theta(w\setminus\srw(w))$ is a rearrangement of $w\setminus\srw(w)$. Moreover $\delta_d(\srw(w))$ is defined as a rearrangement of $\srw(w)$, so this means that $\theta(w)$ is a rearrangement of $w$, and $\theta$ is a well defined map.

To prove that $\theta$ is a bijection we will define its inverse $\eta$ as the map such that:
\begin{itemize}
    \item If $w=\varepsilon$ is the empty word then $\eta(\varepsilon)=\varepsilon$.
    \item Otherwise, if the rightmost hook $\h(w)$ of $w$ is not trivial
    \begin{equation}
        \eta(w)=\rins(\eta(w\setminus \h(w)),\h(w),\lec(\h(w))),
    \end{equation}
    where $w\setminus \h(w)$ is the prefix of $w$ obtained removing $\h(w)$. If $\h(w)$ is a trivial hook, then $\h(w)=w$ and in this case $\eta(w)=w$
\end{itemize}

We can notice that from the definition of $\theta$ we have that if $w$ is increasing then $\srw(w)=w$ and $\theta(w)=w$, and so it is clear that $\theta$ and $\eta$ are inverse when restricted to the subset of increasing words, because they are both the identity in this case.

Now let us consider the case where $w\in W$ is not increasing.

Since $w\setminus\h(w)$ has length strictly less than $w$ and, by definition the special reverse wave of $\eta(w)$ is a rearrangement of $\h(w)$, using inductive reasoning, entirely analogous to what was done previously for $\theta$, we obtain that $\eta(w)$ is a rearrangement of $w$ and that it is a well defined map. 

Let us prove that they are inverse maps by induction on the length $\ell$ of $w$. If $\ell=0$ it is clear. If $\ell>0$ then
\begin{align}
\theta(\eta(w))&=\theta(\rins(\eta(w\setminus\h(w)),\h(w),\lec(\h(w))))\notag\\
&=\theta(\eta(w\setminus\h(w)))*\delta_{\lec(\h(w))}(\h(w))\\
&=(w\setminus\h(w))*\h(w)=w,\notag
\end{align}
where the second equality is true thanks to Proposition~\ref{prop:rinv} and because the special reverse wave of $\rins(\eta(w\setminus\h(w)),\h(w),\lec(\h(w)))$ is a rearrangement of $\h(w)$. The last equality is true thanks to inductive hypothesis, since $w\setminus\h(w)$ has length strictly less than $\ell$, and thanks to the fact that the unique rearrangement of the elements of $\h(w)$ that is a hook with $\lec$ equal to $\lec(\h(w))$ is $\h(w)$ itself. Moreover
\begin{align}
    \eta(\theta(w))&=\rins(\eta(\theta(w)\setminus\h(\theta(w))),\h(\theta(w)),\lec(\h(\theta(w))))\notag\\
    &=\rins(\eta(\theta(w)\setminus\delta_{\srdes(w)}(\srw(w))),\delta_{\srdes(w)}(\srw(w)),\srdes(w))\\
    &=\rins(\eta(\theta(w\setminus\srw(w))),\srw(w),\srdes(w))\notag\\
    &=\rins(w\setminus\srw(w)),\srw(w),\srdes(w))=w\notag,
\end{align}
where the second equality is true since the rightmost hook of $\theta(w)$ is by definition $\delta_{\srdes(w)}(\srw(w))$ and $\lec(\h(\theta(w)))=\srdes(w)$, the third equality is true since $\delta_{\srdes(w)}(\srw(w))$ is a rearrangement of $\srw(w)$ and by the definition of $\rins$ and thanks to the definition of $\theta(w)$. The fourth equality holds by inductive hypothesis, since the length of $w\setminus\srw(w)$ is strictly less than $\ell$ and the last is true thanks to Proposition~\ref{prop:rinv}. We have proved that $\theta$ and $\eta$ are inverse, and so that they are bijections of $W$. Moreover we have already proved that they permute the elements of the words, so that their restrictions to the symmetric group $\mathcal{S}_n$ are two bijections.

We can also show the identity of equation~\eqref{eq:des_lec_2} by an inductive argument on the length $\ell$ of $w\in W$. If $\ell=0$, then $w=\varepsilon$ is the empty word and $\lec(\theta(\varepsilon))=\lec(\varepsilon)=0=\des(\varepsilon$. If $w$ has length $\ell>0$ we have that
\begin{align}
    \lec(\theta(w))&=\lec(\theta(w\setminus\srw(w)))+\lec(\delta_{\srdes(w)}(\srw(w))\notag\\
    &=\des(w\setminus\srw(w))+\srdes(w)=\des(w),
\end{align}
where the first equality holds since $\delta(\srw(w))$ is the rightmost hook of $\theta(w)$, the second is true by inductive hypothesis and by definition of $\delta_{\srdes(w)}$. The last equality is true thanks to Remark~\ref{rem:des_rspezza}.
\end{proof}

\begin{example}
Let $\sigma=45762318\in \mathcal{S}_8\subset W$, then
\begin{align*}
    \theta(45762318)&=\theta(2318)*\delta_2(4576)=\theta(18)*\delta_1(23)*6457\\
    &=\theta(\varepsilon)+\delta_0(18)*32*6457=\varepsilon*18*32*6457 \\
    &=18326457
\end{align*}
and 
\begin{align*}
    \eta(18326457)&=\rins(\eta(1832),6457,2)=\rins(\rins(\eta(18),32,1),6457,2)\\
    &=\rins(\rins(18,32,1),6457,2)=\rins(2318,6457,2)=45762318.
\end{align*}
Moreover $\des(45762318)=3=\lec(18326457)=\lec(\theta(45762318))$.
\end{example}

\begin{remark}\label{rem:cat_bij_2}
    If $w=\gamma*\alpha_1*\cdots*\alpha_k$ is the Hook factorization of a word $w\in W$, then we can write
    \begin{equation*}
        w=\conc(\cdots\conc(\conc(\gamma,\alpha_1),\alpha_2)\cdots     ,\alpha_k),
    \end{equation*}
    where $\conc$ is the binary operation that concatenates words, i.e. $\conc(w',w'')=w'*w''$ for all $w',w''\in W$.

We have that the map $\eta$ is nothing more than the map obtained by replacing each $\conc$ with $\rins$, namely:
\begin{equation*}
    \eta(w)=\rins(\cdots\rins(\rins(\gamma,\alpha_1,\lec(\alpha_1)),\alpha_2,\lec(\alpha_2))\cdots     ,\alpha_k,\lec(\alpha_k)).
    \end{equation*}

Theorem~\ref{thm:bijection_2} therefore also implies that every word in $W$ can be obtained, in a unique manner, through an ordered and finite sequence of r-insertions of reverse waves, starting from an increasing word. The map $\theta$ will then be analogously the one that replaces reverse waves $\beta_i$ with the corresponding hooks $\delta(\beta_i)$ and r-insertion operations with concatenation operations.

In particular, the set of permutations $\sigma\in\mathcal{S}_n$ with $\lec(\sigma)=i$ and whose rightmost hook has length at least $n-k$ is in bijection, via $\eta$ and $\theta$, with the set of permutations $\sigma\in\mathcal{S}_n$ with $\des(\sigma)=i$ and whose reverse special wave has length at least $n-k$. 
\end{remark}

\section{An explicit bijection in the general case of permutations of \texorpdfstring{$k+1$}{k+1} different numbers}
\label{sec:generalexplicit}
In this section we will show an explicit combinatorial bijection between the set of permutations of $k+1$ different numbers in $[n]$ with $i$ descents and the subset of $\mathcal{S}_n$ made of the permutations $\sigma$ such that $\lec(\sigma)=i$ and whose first hook from the right has cardinality at least $n-k$, proving the Corollary~\ref{corollario:bigezionek} in a combinatorial way.

To do this we will construct a bijection that preserves the number of descents between the set of permutations of $k+1$ different numbers in $[n]$ and the subset of $\mathcal{S}_n$ of permutations with special reverse wave of length at least $n-k$. Thanks to the Theorem~\ref{thm:bijection_2} and to the Remark~\ref{rem:cat_bij_2}, composing this bijection with the map $\theta$ of the previous section, gives us the desired bijection.  
\medskip

The first step is the construction of a descent preserving bijection between $\mathcal{S}_{n,k+1}$ and a proper subset of $\mathcal{S}_n$

\begin{theorem}\label{thm:bij_togli_par}
    The map $\mu:\mathcal{S}_{n,k+1}\to\{\sigma\in\mathcal{S}_n\mid \ell(\srw(\sigma))\geq n-k\}$ such that for every $(\alpha_0,\rho)\in\mathcal{S}_{n,k+1}$
    \begin{equation}
        \mu((\alpha_0,\rho))=\alpha_1\alpha_2\cdots\alpha_{n-k-m-1}\alpha_{n-k-1}\alpha_{n-k-2}\cdots\alpha_{n-k-m}\rho_1\cdots\rho_{k+1}
    \end{equation}
    where $(\alpha_0,\rho)=(\{\alpha_1<\dots<\alpha_{n-k-1}\},\rho_1\cdots\rho_{k+1})$ and $m=\#\{\alpha\in\alpha_0\mid\alpha>\rho_1\}$, is a bijection that preserves the number of descents.

    Its inverse is the map $\nu:\{\sigma\in\mathcal{S}_n\mid \ell(\srw(\sigma))\geq n-k\}\to\mathcal{S}_{n,k+1}$ such that for every $\sigma=\sigma_1\cdots\sigma_n$
    \begin{equation}
        \nu(\sigma)=(\{\sigma_1,\dots,\sigma_{n-k-1}\},\sigma_{n-k}\cdots\sigma_n).
    \end{equation}
\end{theorem}

\begin{proof}
If we call $\sigma=\mu((\alpha_0,\rho))$, by definition of the map $\mu$ we have that $$\alpha_1\alpha_2\cdots\alpha_{n-k-m-1}\alpha_{n-k-1}\alpha_{n-k-2}\cdots\alpha_{n-k-m}\rho_1$$ is a prefix of $\sigma$, of length $n-k$, that is also a reverse wave. By Remark~\ref{rem:special_reverse_wave} $\srw(\sigma)$ is the maximal prefix of $\sigma$ that is a reverse wave and so we have that $\srw(\sigma)$ has length at least $n-k$. This implies that the map $\mu$ is well-defined.

In particular the map $\mu$ is the concatenation of the unique reverse wave with elements in $\alpha_0\sqcup\{\rho_1\}$ and $m$ descents with $\rho_2\cdots\rho_{k+1}$.

The number of descents of $(\alpha_0,\rho)$ is by definition
\begin{equation*}
   \des((\alpha_0,\rho))=m+\des(\rho)= \des(\mu((\alpha_0,\rho)))
\end{equation*}
and so we have that $\mu$ preserves the number of descents.

Let us demonstrate that the map $\nu$ is the inverse of $\mu$. We have that
\begin{align*}
    \nu(\mu((\alpha_0,\rho)))&=\nu(\alpha_1\alpha_2\cdots\alpha_{n-k-m-1}\alpha_{n-k-1}\alpha_{n-k-2}\cdots\alpha_{n-k-m}\rho_1\cdots\rho_{k+1})\\
    &=(\{\alpha_1,\dots,\alpha_{n-k-1}\},\rho_1\cdots\rho_{k+1})=(\alpha_0,\rho).
\end{align*}
Moreover if $\sigma=\sigma_1\cdots\sigma_n\in\{\sigma\in\mathcal{S}_n\mid \ell(\srw(\sigma))\geq n-k\}$ and $\srw(\sigma)=\sigma_1\cdots\sigma_s$, with $s\geq n-k$, we have that $\sigma_1\cdots\sigma_{n-k-1}\sigma_{n-k}$ is a prefix of a reverse wave, and so, thanks to Remarks~\ref{rem:pref_rw} and \ref{rem:unique_rwave}, it is the unique reverse wave, with exactly $y=\#\{1\leq i\leq n-k-1\mid\sigma_i>\sigma_{n-k}\}$ descents. Therefore we have that 
\begin{align*}
    \mu(\nu(\sigma))&=\mu((\{\sigma_1,\dots,\sigma_{n-k-1}\},\sigma_{n-k}\sigma_{n-k+1}\cdots\sigma_n)) \\
&=\sigma_1\cdots\sigma_{n-k-1}\sigma_{n-k}*\sigma_{n-k+1}\cdots\sigma_n=\sigma,
\end{align*}
thanks to the previous characterization of $\mu$. 
\end{proof}

\begin{example}
Let $(\alpha_0,\tau)=(\{4,5,7\},62318)\in \mathcal{S}_{8,5}$, then we have that ${m=\{\alpha\in\{4,5,7\}\mid\alpha>6\}=1}$ and so $\mu((\{4,5,7\},62318))=45762318$. Clearly $\nu(45762318)=(\{4,5,7\},62318)$. 

Moreover ${\des((\{4,5,7\},62318))=1+2=3=\des(45762318)}$ and $\ell(\srw(45762318))=\ell(4576)=4\geq4=8-4$.
\end{example}

We have proved, thanks to Remark~\ref{rem:cat_bij_2}, Theorems~\ref{thm:bijection_2} and \ref{thm:bij_togli_par} that the map $\theta\circ\mu:\mathcal{S}_{n,k+1}\to\{\sigma\in\mathcal{S}_n\mid\ell(\h(\sigma))\geq n-k\}$ is a bijection such that, for every $(\alpha_0,\rho)\in\mathcal{S}_{n,k+1}$ we have that $\des((\alpha_0,\rho))=\lec(\theta(\mu((\alpha_0,\rho))))$.

This provides another proof of Corollary~\ref{corollario:bigezionek}.

\section{Second bijection (Foata--Han) preserving des-lec statistic: waves}
\label{sec:foata}
In this section we present a different  bijection $\psi$ of the symmetric group $\mathcal{S}_n$ that sends the $\des$ statistic to the $\lec$ statistic. 

Up to composition with trivial bijections, this bijection $\psi$ appears to coincide with the map $F$ defined by Foata and Han in \cite{FH08}*{Section~6}. In particular, it seems to hold that:
\begin{equation}
    \psi=\mathbf{c}\circ F \circ \mathbf{c}\circ\mathbf{r},
\end{equation}
where $\mathbf{c}(\sigma)=(n+1-\sigma_1)\dots(n+1-\sigma_n)$ and $\mathbf{r}(\sigma)=\sigma_n\dots\sigma_1$. We have verified this equality computationally for $n$ small. We present it anyway because it can be constructed explicitly and directly, using the methods described in the previous sections.

To do so, we will present an explicit bijection of the set $$W=\{a_1\cdots a_{\ell}\mid \ell\in\mathbb{N},\,a_i\in\mathbb{N}^*\,\forall i\text{ and }a_i\neq a_j\,\forall\, i\neq j\}$$ of all finite words with distinct entries, such that if we restrict to the symmetric group $\mathcal{S}_n$, it becomes a bijection of $\mathcal{S}_n$. 

We will define a new way of evaluating descents by decomposition in smaller words.

\begin{definition}
    Let $w=w_1\cdots w_{\ell}\in W$ be a non-empty word of length $\ell$. Let $c$ be the position in $w$ of the maximum element of $w$, we say that $w$ is a \textit{wave} if $c\neq\ell$ and 
    \begin{equation}
        w_{\ell}<w_{\ell-1}<\dots<w_{c+1}<w_1<w_2<\dots<w_c.
    \end{equation}
    We will also denote a word with elements in increasing order, that is the case with $c=\ell$, as a \textit{trivial wave}.
\end{definition}

\begin{remark}
    If $w$ is a wave of length $\ell $ and $c$ is the position of its maximum, then $\des(w)=\ell-c$.
\end{remark}

\begin{definition}\label{def:special_wave}
    Let $w=w_1\cdots w_{\ell}\in W$ be a non-empty word of length $\ell $. The \textit{special wave} of $w$ is the subword $\sw(w)=w_rw_{r+1}\cdots w_t\cdots w_{s-1}w_s$ of $w$ where:
    \begin{itemize}
        \item $1\leq t<\ell$ is the index of the first descent of $w$, so it is the minimum such that $w_t>w_{t+1}$. If $w$ has no descents, if and only if the elements of $w$ are in increasing order, then $\sw(w)=w$.
        \item $1\leq r\leq t$ is the minimum such that $w_r>w_{t+1}$. It exists since $w_t>w_{t+1}$.
        \item $t<s\leq\ell$ is the maximum such that $w_{t+1}>w_{t+2}>\dots>w_{s-1}>w_s>w_{r-1}$, where, if $r=1$, we consider $w_0=0$. It exists since $w_{t+1}>w_{r-1}$, by minimality of $r$, and if $r=1$, then is it true that $w_{t+1}>0=w_{r-1}$ by definition of $W$.
    \end{itemize}    
\end{definition}

\begin{remark}
    If $\sw(w)=w_rw_{r+1}\cdots w_t\cdots w_{s-1}w_s$ is the special wave of $w$, then we have 
    \begin{equation}\label{eq:sw_order}
    w_1<w_2<\dots<w_{r-1}<w_s<w_{s-1}<\dots<w_{t+1}<w_r<\dots<w_t
    \end{equation}
    and in particular that $\sw(w)$ is a wave, since
    \begin{equation}
    w_s<w_{s-1}<\dots<w_{t+1}<w_r<\dots<w_t.
    \end{equation}
    We also have, if $s\neq\ell$, that $w_{s+1}>w_s$ or $w_{s+1}<w_{r-1}$.

    If $w$ has no descents, then $\sw(w)=w$, and in this case $\sw(w)$ is a trivial wave.
\end{remark}

We will denote by $w\setminus\sw(w)=w_1\cdots w_{r-1}w_{s+1}\cdots w_{\ell}$ the word obtained by removing from $w$ its special wave. If $\sw(w)=w$ then $w\setminus\sw(w)$ is the empty word.

\begin{example}\label{ex_sw}
Let $\sigma=13675428\in \mathcal{S}_8\subset W$, then its special wave is $\sw(\sigma)=6754$ and $w\setminus\sw(w)=1328$. From this example we can notice that the special wave is not necessarily maximal, indeed $67542$ is also a wave, but $2<3$.
\end{example}

We can define a map that reverses the operation of removing the special reverse wave.

\begin{proposition}\label{prop:inv}
Let $a=a_1\cdots a_p$ and $b=b_1\cdots b_q$ two words in $W$ with distinct entries and such that $b$ is a non trivial wave. Let $0\leq x\leq p$ be the maximum such that $a_1<a_2<\cdots<a_x<b_q$, where $x=0$ if $a_1>b_q$. Then the map $\ins(a,b)=a_1\cdots a_xb_1\cdots b_qa_{x+1}\cdots a_p$ satisfies
\begin{align}\label{eq:ins1}
    \sw(\ins(a,b))=b,\\
    \ins(a,b)\setminus\sw(\ins(a,b))&=a,\label{eq:ins2}
\end{align}
and, for all $w\in W$ that are not increasing,
\begin{equation}\label{eq:ins3}
    \ins(w\setminus\sw(w),\sw(w))=w.
\end{equation}
\end{proposition}

\begin{proof}
We observe that, since $b$ is a wave and by the definition of $x$, we have that $b_i>b_q>a_x>\cdots>a_1$, for every $1\leq i<q$. Moreover we have by maximality of $x$ that $a_{x+1}>b_q$ or $a_{x+1}<a_x$ or $x=p$. This implies that the first descent of $\ins(a,b)$ is the first descent of $b$, and that the extremal points of its special wave are exactly $b_1$ and $b_q$, and so that $\sw(\ins(a,b))=b$. The equation~\eqref{eq:ins2} is clear from the definition of $\ins$ and from equation~\eqref{eq:ins1}.

To prove the third property, let $w=w_1\cdots w_{\ell}\in W$ be a word and $\sw(w)=w_rw_{r+1}\cdots w_t\cdots w_{s-1}w_s$ its special wave, with $t$ as in the Definition~\ref{def:special_wave}. From equation~\eqref{eq:sw_order} and the fact that if $s\neq\ell$ then $w_{s+1}>w_s$ or $w_{s+1}<w_{r-1}$, we obtain precisely, from the definition, that $\ins(w\setminus\sw(w),\sw(w))=w$.
\end{proof}

\begin{example}
    Let $a=1328$ a word and $b=6754$ a wave. Then $\ins(a,b)=13675428$ and we have already seen in the previous example that $\sw(13675428)=6754=b$. 
\end{example}

Thanks to the definition of the special wave of a word $w$ we can evaluate the statistic $\des$ in a recursive way.

\begin{proposition}\label{prop:des_spezza}
Let $w=w_1\cdots w_{\ell}\in W$ be a word and $\sw(w)=w_rw_{r+1}\cdots w_t\cdots w_{s-1}w_s$ its special wave. Then we have that
\begin{equation}\label{eq:des_spezza}
\des(w)=\des(w\setminus\sw(w))+\des(\sw(w)).
\end{equation}   
\end{proposition}

\begin{proof}
First of all, we note that if $r>1$ and $s<\ell$ we have that
\begin{align}
    \des(w)=\des(w_1\cdots w_{\ell})= &\des(w_1\cdots w_{r-1}w_{s+1}\cdots w_{\ell})+\des(w_r\cdots w_s)+\label{eq:des1}\\&-\chi(w_{r-1}>w_{s+1})+\chi(w_{r-1}>w_{r})+\chi(w_{s}>w_{s+1}),\notag
\end{align}
where $\chi(P)$ equals $1$ if the property $P$ is true, and $0$ otherwise.
From the definition of $r$ we have that $w_{r-1}<w_r$, and so that $\chi(w_{r-1}>w_{r})=0$.

If $w_{s+1}>w_{s}$, then $w_{s+1}>w_{s}>w_{r-1}$ and so that $\chi(w_{r-1}>w_{s+1})=\chi(w_{s}>w_{s+1})=0$.
If $w_{s+1}<w_{s}$, then by definition of $s$, we have that $w_{s+1}<w_{r-1}$ and so that $\chi(w_{r-1}>w_{s+1})=\chi(w_{s}>w_{s+1})=1$. In both cases equation~\eqref{eq:des1} simplifies to equation~\eqref{eq:des_spezza}.

If $r=1$ and $s<\ell$ then
\begin{align*}
    \des(w)=\des(w_1\cdots w_{\ell})&=\des(w_{s+1}\cdots w_{\ell})+\des(w_1\cdots w_s)+\chi(w_{s}>w_{s+1})\\
    &=\des(w_{s+1}\cdots w_{\ell})+\des(w_1\cdots w_s)\\
    &=\des(w\setminus\sw(w))+\des(\sw(w)),
\end{align*}
since by definition of $s$ in this case $w_s<w_{s+1}.$

If $r>1$ and $s=\ell$ then
\begin{align*}
    \des(w)=\des(w_1\cdots w_{\ell})&=\des(w_{1}\cdots w_{r-1})+\des(w_r\cdots w_{\ell})+\chi(w_{r-1}>w_{r})\\
    &=\des(w_{1}\cdots w_{r-1})+\des(w_r\cdots w_{\ell})\\
    &=\des(w\setminus\sw(w))+\des(\sw(w)),
\end{align*}
since by definition of $r$ it holds that $w_{r-1}<w_{r}.$
from the definition of $\sw(w)$, we have that $w_{r-1}<w_{r}$ or $r=1$. 

Finally if $r=1$ and $s=\ell$, then $w=\sw(w)$, $w\setminus\sw(w)$ is the empty word, and then equation~\eqref{eq:des_spezza} is easily true.
\end{proof}

\begin{example}
    Let $\sigma=13675428\in \mathcal{S}_8\subset W$. We have that $$\des(\sigma)=\des(13675428)=3=1+2=\des(1328)+\des(6754)=\des(\sigma\setminus\sw(\sigma))+\des(\sw(\sigma)).$$
\end{example}

\begin{remark}\label{rem:unique_wave}
Similarly to what was observed in Remarks~\ref{lec_inverse} and \ref{rem:unique_rwave}, if a wave $\beta$ has length $\ell $, then we have that $1\leq\des(\beta)\leq\ell-1$. Moreover if we have $\ell$ distinct numbers $b_1<b_2<\cdots<b_{\ell}$, for any $0\leq d\leq\ell-1$, we have exactly one way to rearrange them to create a wave, possibly trivial if $d=0$, $\beta_d$ with $\des(\beta_d)$ equal to $d$, that is, ${\beta_d=b_{d+1}b_{d+2}\cdots b_{\ell}b_db_{d-1}\cdots b_1}$.
\end{remark}

Now we can define the desired bijection.

Let $\psi:W\to W$ be the map such that:
\begin{itemize}
    \item If $w=\varepsilon$ is the empty word then $\psi(\varepsilon)=\varepsilon$.
    \item Otherwise
    \begin{equation}
        \psi(w)=\psi(w\setminus\sw(w))*\delta(\sw(w))
    \end{equation}
    where the $*$ operator is the concatenation of words and $\delta(\sw(w))=c_{d+1}c_1c_2\cdots c_dc_{d+2}\cdots c_m$ is a rearrangement of $\sw(w)$, where $c_1<c_2<\cdots<c_m$ are the elements of $\sw(w)$ and $d=\des(\sw(w))$.
\end{itemize}

\begin{theorem}\label{thm:bijection}
    The map $\psi:W\to W$ just defined is a well defined bijection of $W$, such that \begin{equation}
        \des(w)=\lec(\psi(w)).\label{eq:des_lec}
        \end{equation}

    Moreover the restriction to the symmetric group $\psi_{\mathcal{S}_n}:\mathcal{S}_n\to \mathcal{S}_n$ is a bijection.
\end{theorem}

\begin{proof}
We will prove by induction on the length $\ell$ of $w\in W$ that $\psi(w)$ is a rearrangement of $w$, and that it is a well defined map. If $\ell=0$ then $w=\varepsilon$ is the empty word and by definition $\psi(\varepsilon)=\varepsilon$, that is fine. If $w$ has length $\ell>0$, then by definition $w\setminus\sw(w)$ is a word of length strictly less than $\ell$, and so by the inductive hypothesis $\psi(w\setminus\sw(w))$ is a rearrangement of $w\setminus\sw(w)$. Moreover $\delta(\sw(w))$ is defined as a rearrangement of $\sw(w)$, so this means that $\psi(w)$ is a rearrangement of $w$, and $\psi$ is a well defined map.

To prove that $\psi$ is a bijection we will define its inverse $\varphi$ as the map such that:
\begin{itemize}
    \item If $w=\varepsilon$ is the empty word then $\varphi(\varepsilon)=\varepsilon$.
    \item Otherwise, if the rightmost hook $\h(w)$ of $w$ is not trivial
    \begin{equation}
        \varphi(w)=\ins(\varphi(w\setminus \h(w)),\wv(\h(w))),
    \end{equation}
    where $w\setminus \h(w)$ is the prefix of $w$ obtained removing $\h(w)$ and $$\wv(\h(w))=c_{d+1}c_{d+2}\cdots c_mc_dc_{d-1}\cdots c_1$$ is a rearrangement of $\h(w)$, where $c_1<c_2<\cdots<c_m$ are the elements of $\h(w)$ and $d=\lec(\h(w))=\inv(\h(w))$. If $\h(w)$ is a trivial hook, then $\h(w)=w$ and in this case $\varphi(w)=w$
\end{itemize}
First of all let us note that if $\alpha$ is a hook with $\lec(\alpha)=d$, then, thanks to Remark~\ref{rem:unique_wave}, we have that $\wv(\alpha)$ is the unique wave such that its elements are the same of $\alpha$ and such that $\des(\wv(\alpha))=\lec(\alpha)=d$. Similarly, if $\beta$ is a wave with $\des(\beta)=d$, then, thanks to Remark~\ref{lec_inverse}, we have that $\delta(\beta)$ is the unique hook such that its elements are the same of $\beta$ and such that $\lec(\delta(\beta))=\des(\beta)=d$. Putting these two last observations together, we obtain, if $\alpha$ is a hook and $\beta$ is a wave, that
\begin{align}
    \delta(\wv(\alpha))&=\alpha\label{eq:delta_wv}\\
    \wv(\delta(\beta))&=\beta.\label{eq:wv_delta}
\end{align}

We can notice that from the definition of $\psi$ we have that if $w$ is increasing then $\sw(w)=w$ and $\psi(w)=w$, and so it is clear that $\psi$ and $\varphi$ are inverse when restricted to the subset of increasing words, because they are both the identity in this case.

Now let us consider the case where $w\in W$ is not increasing.

Since $w\setminus\h(w)$ has length strictly less than $w$ and, by definition, $\wv(\h(w))$ is a wave, using inductive reasoning, entirely analogous to what was done previously for $\psi$, we obtain that $\varphi(w)$ is a rearrangement of $w$ and that it is a well defined map. 

Let us prove that they are inverse maps by induction on the length $\ell$ of $w$. If $\ell=0$ it is clear. If $\ell>0$ then
\begin{align}
\psi(\varphi(w))&=\psi(\ins(\varphi(w\setminus\h(w)),\wv(\h(w))))\notag\\
&=\psi(\varphi(w\setminus\h(w)))*\delta(\wv(\h(w)))\\
&=w\setminus\h(w)*\h(w)=w,\notag
\end{align}
where the second equality is true thanks to Proposition~\ref{prop:inv} and the last equality is true thanks to inductive hypothesis, since $w\setminus\h(w)$ has length strictly less than $\ell$, and thanks to equation~\eqref{eq:delta_wv}. Moreover
\begin{align}
    \varphi(\psi(w))&=\ins(\varphi(\psi(w)\setminus\h(\psi(w))),\wv(\h(\psi(w))))\notag\\
    &=\ins(\varphi(\psi(w)\setminus\delta(\sw(w))),\wv(\delta(\sw(w))))\\
    &=\ins(\varphi(\psi(w\setminus\sw(w))),\sw(w))\notag\\
    &=\ins(w\setminus\sw(w)),\sw(w))=w\notag,
\end{align}
where the second equality is true since the rightmost hook of $\psi(w)$ is by definition $\delta(\sw(w))$, the third equality is true thanks to equation~\eqref{eq:wv_delta} and thanks to the definition of $\psi(w)$. The fourth equality holds by inductive hypothesis, since the length of $w\setminus\sw(w)$ is strictly less than $\ell$ and the last is true thanks to Proposition~\ref{prop:inv}. We have proved that $\psi$ and $\varphi$ are inverse, and so that they are bijections of $W$. Moreover we have already proved that they permute the elements of the words, so that their restrictions to the symmetric group $\mathcal{S}_n$ are two bijections.

We can also show the identity of equation~\eqref{eq:des_lec} by an inductive argument on the length $\ell$ of $w\in W$. If $\ell=0$, then $w=\varepsilon$ is the empty word and $\lec(\psi(\varepsilon))=\lec(\varepsilon)=0=\des(\varepsilon$. If $w$ has length $\ell>0$ we have that
\begin{align}
    \lec(\psi(w))&=\lec(\psi(w\setminus\sw(w)))+\lec(\delta(\sw(w))\notag\\
    &=\des(w\setminus\sw(w))+\des(\sw(w))=\des(w),
\end{align}
where the first equality holds since $\delta(\sw(w))$ is the rightmost hook of $\psi(w)$, the second is true by inductive hypothesis and by definition of $\delta$. The last equality is true thanks to Proposition~\ref{prop:des_spezza}.
\end{proof}

\begin{example}
Let $\sigma=13675428\in \mathcal{S}_8\subset W$, then
\begin{align*}
    \psi(13675428)&=\psi(1328)*\delta(6754)=\psi(18)*\delta(32)*6457\\
    &=\psi(\varepsilon)+\delta(18)*32*6457=\varepsilon*18*32*6457 \\
    &=18326457
\end{align*}
and 
\begin{align*}
    \varphi(18326457)&=\ins(\varphi(1832),\wv(6457))=\ins(\ins(\varphi(18),\wv(32)),6754)\\
    &=\ins(\ins(18,32),6754)=\ins(1328,6754)=13675428.
\end{align*}
Moreover $\des(13675428)=3=\lec(18326457)=\lec(\psi(13675428))$.
\end{example}

\begin{remark}\label{rem:cat_bij}
    If $w=\gamma*\alpha_1*\cdots*\alpha_k$ is the Hook factorization of a word $w\in W$, then we can write
    \begin{equation*}
        w=\conc(\cdots\conc(\conc(\gamma,\alpha_1),\alpha_2)\cdots     ,\alpha_k),
    \end{equation*}
    where $\conc$ is the binary operation that concatenates words, i.e. $\conc(w',w'')=w'*w''$ for all $w',w''\in W$.

We have that the map $\varphi$ is nothing more than the map obtained by replacing each $\conc$ with $\ins$, and each hook $\alpha_i$ with its corresponding wave $\wv(\alpha_i)$, namely:
\begin{equation*}
    \varphi(w)=\ins(\cdots\ins(\ins(\gamma,\wv(\alpha_1)),\wv(\alpha_2))\cdots     ,\wv(\alpha_k)).
    \end{equation*}

Theorem~\ref{thm:bijection} therefore also implies that for every word in $W$, it can be obtained, in a unique manner, through an ordered and finite sequence of insertions of waves, starting from an increasing word. The map $\psi$ will then be analogously the one that replaces waves $\beta_i$ with the corresponding hooks $\delta(\beta_i)$ and insertion operations with concatenation operations.

In particular, the set of permutations $\sigma\in\mathcal{S}_n$ with $\lec(\sigma)=i$ and whose rightmost hook has length at least $n-k$ is in bijection, via $\varphi$ and $\psi$, with the set of permutations $\sigma\in\mathcal{S}_n$ with $\des(\sigma)=i$ and whose special wave has length at least $n-k$. 
\end{remark}

This result improves Corollary~\ref{corollario:bigezionek}:

\begin{theorem}\label{thm:bigezionek}
Let $0\leq k\leq n-1$ and $i\geq0$. The following sets have the same cardinality:
\begin{itemize}
    \item the set of the permutations of $k+1$ different numbers in $[n]$ with $i$ descents,
    \item the set of permutations $\sigma\in\mathcal{S}_n$ with $i$ descents and whose special reverse wave has length at least $n-k$,
    \item the set of permutations $\sigma\in\mathcal{S}_n$ with $i$ descents and whose special wave has length at least $n-k$,
    \item the set of permutations $\sigma\in\mathcal{S}_n$ with $\lec(\sigma)=i$ and whose rightmost hook has length at least $n-k$.
\end{itemize}
\end{theorem}

\begin{remark}
    The cardinality of the previous sets, when $i>0$, is
    \begin{equation}
        \sum_{d=n-k}^n\binom{n}{d}\sum_{\ell=1}^{d-1}E(n-d,i-\ell),
    \end{equation}
    where $E(a,b)=\#\{\sigma\in\mathcal{S}_a\mid\lec(\sigma)=b\}$ are the Eulerian numbers.

    Indeed, from the third characterization, we can count these permutations by first choosing the $d$ elements and the value $\ell$ of the $\lec$ of the rightmost hook. 

    We can also rephrase this as

    \begin{equation}
        \sum_{d=0}^k\binom{n}{d}\sum_{\ell=1}^{n-d-1}E(d,i-\ell).
    \end{equation}
\end{remark}

\bibliographystyle{amsalpha}
\bibliography{Biblebib}
\end{document}